\documentclass[11pt]{amsart}
\usepackage{amsfonts}
\usepackage{graphicx}
\usepackage{amssymb}
\usepackage{amsmath}
\usepackage{amscd}
\usepackage{mathrsfs}
\usepackage{stmaryrd}
\usepackage{bbm}
\usepackage{longtable}
\usepackage{txfonts}
\usepackage{wrapfig}
\usepackage{picins}
\usepackage{float}
\usepackage{xypic}
\usepackage{dsfont}
\usepackage{xcolor}
\usepackage{color}
\usepackage[colorlinks,linkcolor=blue,anchorcolor=blue,
citecolor=blue]{hyperref}
\usepackage{hyperref}
\allowdisplaybreaks

\newtheorem{theorem}{Theorem}[subsection]
\newtheorem{prop}[theorem]{Proposition}

\newtheorem{lem}[theorem]{Lemma}

\newtheorem{rem}[theorem]{Remark}
\newtheorem{conj}{Conjecture}
\newtheorem{exam}[theorem]{Example}

\begin{document}
\setlength{\oddsidemargin}{0cm}
\setlength{\evensidemargin}{0cm}

\title{On Modular Invariants  of A Vector and A Covector}

\author[ychen]{Yin Chen}

\address[1]{
Chern Institute of Mathematics, Nankai University, Tianjin 300071, P.R. China
}
\address[2]{
School of Mathematics and Statistics, Northeast Normal University,
Changchun 130024, P.R. China
}

\email{ychen@nenu.edu.cn}

\date{\today}

\def\shorttitle{Invariants of a vector and a covector}

\begin{abstract}
Let $SL_{2}(F_{q})$ be the special linear group over a finite field $F_{q}$, $V$ be the 2-dimensional natural representation of $SL_{2}(F_{q})$ and $V^{\ast}$ be the dual representation.
We denote by $F_{q}[V\oplus V^{\ast}]^{SL_{2}(F_{q})}$  the corresponding invariant ring of a vector and a covector for  $SL_{2}(F_{q})$. In this paper, we construct a free module basis over some homogeneous system of parameters of $F_{q}[V\oplus V^{\ast}]^{SL_{2}(F_{q})}$. We calculate the Hilbert series of $F_{q}[V\oplus V^{\ast}]^{SL_{2}(F_{q})}$, and prove that it is a Gorenstein algebra. As an application, we confirm  a special case of the recent conjecture of Bonnaf\'{e} and  Kemper  in \cite{BK2011}.
\end{abstract}

\subjclass[2010]{13A50, 20H30.}

\keywords{Modular invariant; Gorenstein algebra; Hilbert series; Homogeneous system of parameters}

\maketitle
\baselineskip=15pt

\section{Introduction}
\setcounter{equation}{0}
\renewcommand{\theequation}
{1.\arabic{equation}}

\setcounter{theorem}{0}
\renewcommand{\thetheorem}
{1.\arabic{theorem}}

Let $K$ be a field and $V$ be an $n$-dimensional faithful representation of a finite group $G$ over $K$.
Let $V^{\ast}$ be the dual representation of $V$ and  $K[V\oplus V^{\ast}]$ be the symmetric algebra of $V^{\ast}\oplus V$.
If we choose $\{x_{1},\cdots,x_{n}\}$ as a standard basis of $V$ and $\{y_{1},\cdots,y_{n}\}$ as the dual basis, then
$K[V\oplus V^{\ast}]$  can be identified with the polynomial ring $K[x_{1},\cdots,x_{n},y_{1},\cdots,y_{n}]$.
For every element  $g\in G$, the action of $g$ on $V$ induces an  action  on $K[V\oplus V^{\ast}]$ as a $K$-algebra automorphism.
In the language of classical invariant theory,
the subring consisting  of all $G$-invariant polynomials is called
the \textit{invariant ring of a vector and a covector} of $G$  and denoted by $K[V\oplus V^{\ast}]^{G}$. If the characteristic of $K$ divides the order of $G$ then the invariant ring $K[V\oplus V^{\ast}]^{G}$ is called \textit{modular}; otherwise \textit{non-modular} or \textit{ordinary}.
One of the purposes in invariant theory is to find  explicit generators and relations among them for the invariant ring  $K[V\oplus V^{\ast}]^{G}$. We recommend \cite{Ben1993,CW2010,DK2002,NS2002}
as general references for invariant theory of finite groups.

In this paper, we shall study the modular invariant ring $K[V\oplus V^{\ast}]^{G}$ for some finite linear group $G$ with its natural modular $2$-dimensional representation $V$.  There are two motivations which lead us to consider the modular case.
Firstly, the non-modular invariant ring $K[V\oplus V^{\ast}]^{G}$ is of importance in representation theory of Cherednik algebras and the geometry of Hilbert schemes. Especially in the situation where $k$ is the field of complex numbers and $G$ is a finite reflection group, it has been studied intensively and affected in many research papers over the last twenty years (see \cite{Gor2003,Hai1994,Hai2002,BK2011} and references therein). However, very few results seem to be known in the modular case. The second reason comes from the modular vector invariant theory. A well-known fact asserts that $F_{q}[V\oplus V]^{U_{n}(F_{q})}$, the ring of vector invariants of 2 copies of $V$,
 is not Cohen-Macaulay where $n\geq 3$ and $V$ is the natural module of $U_{n}(F_{q})$, the group of all upper triangular  matrices  over a finite field $F_{q}$. On the other hand, a recent result due to Bonnaf\'{e} and Kemper \cite{BK2011} asserts that the invariant ring  $F_{q}[V\oplus V^{\ast}]^{P_{n}(F_{q})}$
 is complete intersection and so Cohen-Macaulay. This  interesting example indicates that  there may be a big difference between
$F_{q}[V\oplus V^{\ast}]^{G}$ and $F_{q}[V\oplus V]^{G}$ for the same group $G$ and the same representation $V$.

In what follows we assume that $K=F_{q}$ is a finite field with $q=p^{r}$ elements, where $p$ is a prime and $r\geq 1$ is a positive number.
We denote by $U_{n}(F_{q})$ the group of all upper triangular  matrices in the special linear group $SL_{n}(F_{q})$ and by
$P_{n}(F_{q})$  the group of all upper triangular matrices with 1's on the main diagonal. Recall that $P_{n}(F_{q})$ is the Sylow $p$-subgroup of $U_{n}(F_{q})$, $SL_{n}(F_{q})$ and the general linear group $GL_{n}(F_{q})$.
In the recent paper \cite{BK2011} by Bonnaf\'{e} and Kemper,  the invariant rings  $F_{q}[V\oplus V^{\ast}]^{U_{n}(F_{q})}$ and $F_{q}[V\oplus V^{\ast}]^{P_{n}(F_{q})}$
were calculated, and  the conclusions assert that two invariant rings  are both complete intersection, where $V$ is the natural representation. In last section of the same paper, the authors have raised the following conjecture.
\begin{conj}[Bonnaf\'{e}-Kemper \cite{BK2011}]
The invariant ring $F_{q}[V\oplus V^{\ast}]^{GL_{n}(F_{q})}$
is a Gorenstein algebra.
\end{conj}
 \noindent Bonnaf\'{e} and Kemper verified this conjecture for some special $n$ and $q$ by various computations in the computer algebra system \verb"magma". However, they did not give any mathematical proof for this conjecture (even in the case $n=2$).

As we know, a basic and very complicated task in modular invariant theory is to determine the structure of the invariant ring $F_{q}[V\oplus V^{\ast}]^{G}$ for a finite group $G$ and a modular representation $V$. We consider $G=SL_{2}(F_{p})$ and $F_{q}=F_{p}$ is a prime field.
Since the 2-dimensional natural representation $V_{2}$ is self-dual, the invariant ring
$F_{p}[V_{2}\oplus V_{2}^{\ast}]^{SL_{2}(F_{p})}$ is isomorphic to the vector invariant ring $F_{p}[2V_{2}]^{SL_{2}(F_{p})}$.  Recently, the latter was computed by Campell \textit{et al}. in \cite{CSW2010}. (Actually, a set of minimal generators for $F_{p}[mV_{2}]^{SL_{2}(F_{p})}$ was constructed.)
Apart from this, it seems there are no
more advances in investigating the structure of the invariant ring $F_{q}[V\oplus V^{\ast}]^{G}$ for a finite linear group $G$ with its natural module $V$,  even for the most important groups $SL_{n}(F_{q})$ and $GL_{n}(F_{q})$.

The main purpose of this paper is to construct a free module basis over some homogeneous system of parameters for $F_{q}[V\oplus V^{\ast}]^{SL_{2}(F_{q})}$ and $F_{q}[V\oplus V^{\ast}]^{GL_{2}(F_{q})}$ respectively. Some calculations for the Hilbert series of the two rings of invariants allow us to show that they are Gorenstein algebras. As an application, we confirm  a special case ($n=2$) of the above conjecture.

For simplicity, in the rest of this paper we write $P_{2},SL_{2}$, and $GL_{2}$ for $P_{2}(F_{q})$, $SL_{2}(F_{q})$, and $GL_{2}(F_{q})$  respectively. The symbols $P_{2}^{2}$, $SL_{2}^{2}$ and $GL_{2}^{2}$ will be employed to denote the products $P_{2}\times P_{2}$, $SL_{2}\times SL_{2}$ and $GL_{2}\times GL_{2}$ respectively. We write $R_{2}$ for the polynomial ring $F_{q}[x_{1},x_{2},y_{1},y_{2}]$.
The result of Bonnaf\'{e} and Kemper \cite{BK2011} asserts that the invariant ring $R_{2}^{P_{2}}$
is a hypersurface and so  is Cohen-Macaulay.  By a theorem of Campbell, Hughes and Pollack (Theorem 1 in \cite{CHP1991}),
the invariant rings  $R_{2}^{SL_{2}}$
and $R_{2}^{GL_{2}}$ are both Cohen-Macaulay.
This fact leads us to construct a free module basis for $R_{2}^{SL_{2}}$
and $R_{2}^{GL_{2}}$ over a suitably chosen homogeneous system of parameters.
The essence of this idea is the same as
that of Campbell and Hughes in \cite{CH1996}. The second step is to compute the corresponding  Hilbert series $H\left(R_{2}^{SL_{2}},t\right)$
and $H\left(R_{2}^{GL_{2}}, t\right)$. Finally, we apply a well-known theorem due to Stanley \cite{Sta1978} to prove the Gorensteinness of $R_{2}^{SL_{2}}$
and $R_{2}^{GL_{2}}$.

This paper is organized as follows. Section 2 is a preliminary  which contains some basic facts and constructions; more precisely, we construct a free module basis of $R_{2}^{P_{2}}$ over $R_{2}^{SL_{2}^{2}}$.
In Section 3, we extend some results of Campbell and Hughes \cite{CH1996} from  a prime field $F_{p}$ to  a  general finite field $F_{q}$ with $q=p^{r}$;
we construct a free module basis of $R_{2}^{SL_{2}}$ over $R_{2}^{SL_{2}^{2}}$ and compute the Hilbert series of  $R_{2}^{SL_{2}}$; we deduce that $R_{2}^{SL_{2}}$ is a Gorenstein algebra.
In Section 4, we apply the same techniques  to the case of $GL_{2}$ and  prove that $R_{2}^{GL_{2}}$  is also  a Gorenstein algebra.

\section{Preliminaries}
\setcounter{equation}{0}
\renewcommand{\theequation}
{2.\arabic{equation}}

\setcounter{theorem}{0}
\renewcommand{\thetheorem}
{2.\arabic{theorem}}

We use the same abbreviations as  in the Introduction,  such as $P_{2}$, $SL_{2}$, $GL_{2}$, and $R_{2}$.
For every $2\times 2$ nonsingular matrix $A$, the map $A\mapsto\left[\begin{smallmatrix}
(A^{t})^{-1}&0\\
0&A
\end{smallmatrix}\right]$ gives an embedding of $GL_{2}$ into $GL_{4}$, where $(A^{t})^{-1}$ is the inverse of the transpose of $A$.
Thus the groups $P_{2}$, $SL_{2}$, and $GL_{2}$ can be viewed as  subgroups of $P_{2}^{2}$,
$SL_{2}^{2}$, and $GL_{2}^{2}$ respectively.
We write $A=\left[\begin{smallmatrix}
a&b\\
c&d
\end{smallmatrix}\right]$, the induced  action of $A$ on $R_{2}$ is given by
\begin{eqnarray}
A\cdot \left[\begin{smallmatrix}
x_{1}\\
x_{2}
\end{smallmatrix}\right]
&=&(\textrm{det}A)^{-1}\left[\begin{smallmatrix}
d&-c\\
-b&a
\end{smallmatrix}\right]\left[\begin{smallmatrix}
x_{1}\\
x_{2}
\end{smallmatrix}\right]=(\textrm{det}A)^{-1}\left[\begin{smallmatrix}
dx_{1}-cx_{2}\\
-bx_{1}+ax_{2}
\end{smallmatrix}\right]\label{eq:2.1}\\
A\cdot \left[\begin{smallmatrix}
y_{1}\\
y_{2}
\end{smallmatrix}\right]&=&\left[\begin{smallmatrix}
a&b\\
c&d
\end{smallmatrix}\right]\left[\begin{smallmatrix}
y_{1}\\
y_{2}
\end{smallmatrix}\right]=\left[\begin{smallmatrix}
ay_{1}+by_{2}\\
cy_{1}+dy_{2}
\end{smallmatrix}\right].\label{eq:2.2}
 \end{eqnarray}

We define an involution $\ast: R_{2}\rightarrow R_{2}$ given by $x_{1}\mapsto y_{2}, x_{2}\mapsto y_{1}, y_{1}\mapsto x_{2}, y_{2}\mapsto x_{1}$. Moreover, for  $G \in\left\{P_{2},SL_{2}, GL_{2}\right\}$,  the involution $\ast$ induces an automorphism of the invariant ring $R_{2}^{G}$.
Recall that the orders of groups $P_{2}$, $SL_{2}$, and $GL_{2}$ are  $q$,  $q(q^{2}-1)$,  and  $(q^{2}-1)(q^{2}-q)$ respectively.
A theorem of Kemper (\cite{Kem1996}, Proposition 16) implies that the invariant ring
$R_{2}^{SL_{2}^{2}}$ is a polynomial algebra and generated by the Dickson invariants and their $\ast$-images:
\begin{eqnarray*}
d_{2,2}=\textrm{det}\left[\begin{smallmatrix}
x_{2}&x_{2}^{q}\\
x_{1}&x_{1}^{q}
\end{smallmatrix}\right], & ~~c_{2,1}=(d_{2,2})^{-1}\cdot\textrm{det}\left[\begin{smallmatrix}
x_{2}&x_{2}^{q^{2}}\\
x_{1}&x_{1}^{q^{2}}
\end{smallmatrix}\right];\\
d_{2,2}^{\ast}=\textrm{det}\left[\begin{smallmatrix}
y_{1}&y_{1}^{q}\\
y_{2}&y_{2}^{q}
\end{smallmatrix}\right], & ~~c_{2,1}^{\ast}=(d_{2,2}^{\ast})^{-1}\cdot\textrm{det}\left[\begin{smallmatrix}
y_{1}&y_{1}^{q^{2}}\\
y_{2}&y_{2}^{q^{2}}
\end{smallmatrix}\right].
 \end{eqnarray*}

 The following two lemmas are useful in our discussion (see the proof of Proposition 3.1 in \cite{CH1996} and  Theorem 1 in \cite{CHP1991} for the detailed  proofs).

 \begin{lem} \label{lem:2.1}
Let $R$ be  the polynomial ring over a field $K$ and $G$ act on $R$ as a $K$-automorphism group.
If $H$ is  a subgroup of $G$  such that the subring of fixed elements $R^{H}$ is Cohen-Macaulay and $R^{G}$ a polynomial ring, then
$R^{H}$ is a free module over $R^{G}$ of rank $[G:H]$.
\end{lem}

\begin{lem}\label{prop:2.2}
Let $P$ be a Sylow $p$-subgroup of a finite group $G$ acting linearly on $R$, a polynomial algebra  over a field.  If the  invariant ring $R^{P}$ is Cohen-Macaulay,
then also is $R^{G}$.
\end{lem}

 \begin{prop} \label{prop:2.3} The invariant rings $R_{2}^{P_{2}}$ and $R_{2}^{SL_{2}}$ are  free modules over $R_{2}^{SL_{2}^{2}}$ with
the ranks $q(q^{2}-1)^{2}$ and  $q(q^{2}-1)$ respectively.
\end{prop}

\begin{proof} By the result of Bonnaf\'{e} and Kemper (\cite{BK2011}, Theorem 2.4),
$R_{2}^{P_{2}}$ is a hypersurface and so Cohen-Macaulay. Since the group $P_{2}$ is a Sylow $p$-subgroup of $SL_{2}$,  it follows from Lemma \ref{prop:2.2}
that  $R_{2}^{SL_{2}}$ is also Cohen-Macaulay. Now the invariant ring $R_{2}^{SL_{2}^{2}}$ is a polynomial algebra generated by $\left\{d_{2,2},c_{2,1},d_{2,2}^{\ast},c_{2,1}^{\ast}\right\}$, and
$\left\{d_{2,2},c_{2,1},d_{2,2}^{\ast},c_{2,1}^{\ast}\right\}$ is a homogeneous system of parameters for $R_{2}^{P_{2}}$ and $R_{2}^{SL_{2}}$, thus
they are free modules over $R_{2}^{SL_{2}^{2}}$. By Lemma \ref{lem:2.1}, the corresponding  ranks are equal to $q(q^{2}-1)^{2}$ and  $q(q^{2}-1)$ respectively.
\end{proof}

  We define $\phi_{1}:=x_{1}, \phi_{2}:=x_{2}^{q}-x_{2}x_{1}^{q-1},$ and $u_{0}:=x_{1}y_{1}+x_{2}y_{2}$. Then $\phi_{1}^{\ast}=y_{2}$, $\phi_{2}^{\ast}=y_{1}^{q}-y_{1}y_{2}^{q-1}$, and $u_{0}^{\ast}=u_{0}$.
In \cite{BK2011}, it is proved that the invariant ring
$R_{2}^{P_{2}}$ is a hypersurface and generated by
$\{\phi_{1},\phi_{2},\phi_{1}^{\ast},\phi_{2}^{\ast},u_{0}\}$ with the single relation
\begin{equation}\label{eq:2.3}
u_{0}^{q}=(\phi_{1}\phi_{2})^{q-1}u_{0}+\phi_{1}^{q}\phi_{2}^{\ast}+(\phi_{1}^{\ast})^{q}\phi_{2}.
\end{equation}
It is easy to verify  that
\begin{eqnarray}
\phi_{1}\phi_{2}&=&-d_{2,2}\label{eq:2.4}\\
\phi_{1}^{\ast}\phi_{2}^{\ast}&=&-d_{2,2}^{\ast}\label{eq:2.5}\\
\phi_{1}^{q(q-1)+1}&=&\phi_{1}c_{2,1}-d_{2,2}\phi_{2}^{q-2}\label{eq:2.6}\\
(\phi_{1}^{\ast})^{q(q-1)+1}&=&(\phi_{1}^{\ast})c_{2,1}^{\ast}-d_{2,2}^{\ast}(\phi_{2}^{\ast})^{q-2}\label{eq:2.7}\\
\phi_{2}^{q-1}&=&-\phi_{1}^{q(q-1)}+c_{2,1}\label{eq:2.8}\\
(\phi_{2}^{\ast})^{q-1}&=&-(\phi_{1}^{\ast})^{q(q-1)}+c_{2,1}^{\ast}\label{eq:2.9}.
\end{eqnarray}

Put \begin{equation}
\mathds{P}=\left\{ \begin{aligned}
\phi_{1}^{i}(\phi_{1}^{\ast})^{j}u_{0}^{k},&\quad 0\leq i,j\leq q(q-1),~~ 0\leq k\leq q-1;\\
\phi_{1}^{i}(\phi_{2}^{\ast})^{j}u_{0}^{k},&\quad 0\leq i\leq q(q-1),~~ 1\leq j\leq q-2, ~~0\leq k\leq q-1;\\
(\phi_{1}^{\ast})^{i}\phi_{2}^{j}u_{0}^{k},&\quad 0\leq i\leq q(q-1),~~ 1\leq j\leq q-2,~~ 0\leq k\leq q-1;\\
\phi_{2}^{i}(\phi_{2}^{\ast})^{j}u_{0}^{k},&\quad 1\leq i,j\leq q-2,~~ 0\leq k\leq q-1
 \end{aligned} \right\}.
 \end{equation}

 \begin{prop} \label{prop:2.4}
All elements in the set $\mathds{P}$ form a free $R_{2}^{SL_{2}^{2}}$-module basis of
$R_{2}^{P_{2}}$.
\end{prop}

 \begin{proof}
No loss of generality, we assume that a free $R_{2}^{SL_{2}^{2}}$-module  basis
for $R_{2}^{P_{2}}$ consists of monomials $\phi_{1}^{i}(\phi_{1}^{\ast})^{j}\phi_{2}^{r}(\phi_{2}^{\ast})^{s}u_{0}^{t}$.
Let $I_{P}$ be the homogeneous ideal  generated by $\left\{d_{2,2},c_{2,1},d_{2,2}^{\ast},c_{2,1}^{\ast}\right\}$ in $R_{2}^{P_{2}}$.
We see that in (\ref{eq:2.4})-(\ref{eq:2.7}), each monomial on its left-hand side belong to $I$ and so none of these monomials
and their powers will
appear in a free basis. The (\ref{eq:2.8}) and (\ref{eq:2.9}) indicate that one of $\phi_{2}^{q-1}$ and $\phi_{1}^{q(q-1)}$
$\left((\phi_{2}^{\ast})^{q-1}\textrm{ and }(\phi_{1}^{\ast})^{q(q-1)}\right)$ may appear in a free basis but not both. We choose $\phi_{2}^{q-1}$ and $\phi_{1}^{q(q-1)}$
as the elements in a free basis. Moreover, it follows from (\ref{eq:2.3}) that the powers of $u_{0}$ larger than $q-1$ do not appear in a free basis.
The set of all monomials not eliminated from a free basis is just $\mathds{P}$.  We observe that $\mathds{P}$ has exactly
 $q(q^{2}-1)^{2}$ elements, thus it must be a free basis.
\end{proof}

 \section{$SL_{2}$-Invariants}
\setcounter{equation}{0}
\renewcommand{\theequation}
{3.\arabic{equation}}

\setcounter{theorem}{0}
\renewcommand{\thetheorem}
{3.\arabic{theorem}}

The  purpose of this section is to construct a free $R_{2}^{SL_{2}^{2}}$-module basis of  $R_{2}^{SL_{2}}$  and to compute
the corresponding  Hilbert series $H\left(R_{2}^{SL_{2}},t\right)$. We define
\begin{equation}\label{eq:3.1}
u_{1}:=x_{1}^{q}y_{1}+x_{2}^{q}y_{2}.
 \end{equation} Then  $u_{1}^{\ast}=x_{1}y_{1}^{q}+x_{2}y_{2}^{q}.$
 By (\ref{eq:2.1}) and (\ref{eq:2.2}), we see that $u_{1}, u_{0}$, and $u_{1}^{\ast}$ are in $R_{2}^{GL_{2}}$ and thus  belong to $R_{2}^{SL_{2}}$. Moreover, it is not
difficult to check the following identities from the definitions:
\begin{align}
u_{0}^{q+1}&=u_{1}u_{1}^{\ast}-d_{2,2}d_{2,2}^{\ast}\label{eq:3.2}\\
u_{1}^{q}&=c_{2,1}u_{0}^{q}-(d_{2,2})^{q-1}u_{1}^{\ast}\label{eq:3.3}\\
(u_{1}^{\ast})^{q}&=c_{2,1}^{\ast}u_{0}^{q}-(d_{2,2}^{\ast})^{q-1}u_{1}\label{eq:3.4}\\
u_{1}&=\phi_{1}^{q-1}u_{0}+\phi_{1}^{\ast}\phi_{2}\label{eq:3.5}\\
u_{1}^{\ast}&=(\phi_{1}^{\ast})^{q-1}u_{0}+\phi_{1}\phi_{2}^{\ast}\label{eq:3.6}.
\end{align}

 We also define
\begin{equation}\label{eq:3.7}
h_{s}:=\frac{u_{1}^{s+1}(d_{2,2}^{\ast})^{q-s-1}+(u_{1}^{\ast})^{q-s}d_{2,2}^{s}}{u_{0}^{q}}\quad(s=0,1,\cdots,q-1).
\end{equation} Obviously, $h_{s}^{\ast}=h_{q-1-s}$ for all $s$.
In particular,  $h_{0}=c_{2,1}^{\ast}$ and $h_{q-1}=c_{2,1}$.

 \begin{lem}\label{lem:3.1}
Each $h_{s}$ is an $SL_{2}$-invariant polynomial.
\end{lem}

 \begin{proof} It follows from (\ref{eq:3.4}) that $u_{1}^{s+1}(d_{2,2}^{\ast})^{q-1}=c_{2,1}^{\ast}u_{0}^{q}u_{1}^{s}-(u_{1}^{\ast})^{q}u_{1}^{s}$. Thus
\begin{eqnarray*}
h_{s}&=&\frac{u_{1}^{s+1}(d_{2,2}^{\ast})^{q-s-1}+(u_{1}^{\ast})^{q-s}d_{2,2}^{s}}{u_{0}^{q}}\\
&=&\frac{u_{1}^{s+1}(d_{2,2}^{\ast})^{q-1}+(u_{1}^{\ast})^{q-s}(d_{2,2}d_{2,2}^{\ast})^{s}}{u_{0}^{q}(d_{2,2}^{\ast})^{s}}\\
&=&\frac{c_{2,1}^{\ast}u_{0}^{q}u_{1}^{s}-(u_{1}^{\ast})^{q}u_{1}^{s}+(u_{1}^{\ast})^{q-s}(d_{2,2}d_{2,2}^{\ast})^{s}}{u_{0}^{q}(d_{2,2}^{\ast})^{s}}.
\end{eqnarray*}
By (\ref{eq:3.2}) we have
\begin{eqnarray*}
(d_{2,2}d_{2,2}^{\ast})^{s}&=&\left(u_{1}u_{1}^{\ast}-u_{0}^{q+1}\right)^{s}=\sum_{k=0}^{s}(-1)^{s-k}{s\choose k}(u_{1}u_{1}^{\ast})^{k}\left(u_{0}^{q+1}\right)^{s-k}.
 \end{eqnarray*}
Hence
 \begin{eqnarray}\label{eq:3.8}
h_{s}=\frac{c_{2,1}^{\ast}u_{0}^{q}u_{1}^{s}+(u_{1}^{\ast})^{q-s}\left(\sum_{k=0}^{s-1}(-1)^{s-k}{s\choose k}(u_{1}u_{1}^{\ast})^{k}\left(u_{0}^{q+1}\right)^{s-k}\right)}{u_{0}^{q}(d_{2,2}^{\ast})^{s}}.
 \end{eqnarray}
Every term in the numerator of $h_{s}$ contains $u_{0}^{q}$ as a factor and
$u_{0}^{q}$ is prime to $(d_{2,2}^{\ast})^{s}$. Thus the numerator of $h_{s}$ is divisible by $u_{0}^{q}(d_{2,2}^{\ast})^{s}$, i.e.,
$h_{s}$ is a polynomial. The $SL_{2}$-invariance of $h_{s}$ holds immediately.
\end{proof}

 \begin{lem}\label{lem:3.2}
For all $s=0,1,\cdots,q-1$, the following identities hold:
\begin{equation}\label{eq:3.9}
\begin{split}
u_{1}^{\ast}h_{s}&=u_{0}u_{1}^{s}(d_{2,2}^{\ast})^{q-s-1}+d_{2,2}h_{s-1}\\
u_{1}h_{s}&=u_{0}(u_{1}^{\ast})^{q-s-1}d_{2,2}^{s}+d_{2,2}^{\ast}h_{s+1}.
\end{split}
 \end{equation}
In particular,  we have
\begin{eqnarray}
u_{0}(u_{1}^{\ast})^{q-1}&=&u_{1}h_{0}-d_{2,2}^{\ast}h_{1},\label{eq:3.10}\\
u_{0}u_{1}^{q-1}&=&u_{1}^{\ast}h_{q-1}-d_{2,2}h_{q-2}.\label{eq:3.11}
 \end{eqnarray}
\end{lem}

 \begin{proof}
Indeed, it follows from (\ref{eq:3.7}) that
\begin{eqnarray*}
u_{1}^{\ast}h_{s}&=&\frac{u_{1}^{\ast}u_{1}^{s+1}(d_{2,2}^{\ast})^{q-s-1}+(u_{1}^{\ast})^{q-s+1}d_{2,2}^{s}}{u_{0}^{q}}\\
d_{2,2}h_{s-1}&=&\frac{d_{2,2}u_{1}^{s}(d_{2,2}^{\ast})^{q-s}+(u_{1}^{\ast})^{q-s+1}d_{2,2}^{s}}{u_{0}^{q}}.
 \end{eqnarray*}
Thus
\begin{eqnarray*}
u_{1}^{\ast}h_{s}-d_{2,2}h_{s-1}&=&\frac{u_{1}^{\ast}u_{1}^{s+1}(d_{2,2}^{\ast})^{q-s-1}-d_{2,2}u_{1}^{s}(d_{2,2}^{\ast})^{q-s}}{u_{0}^{q}}\\
&=&\frac{(u_{1}^{\ast}u_{1})u_{1}^{s}(d_{2,2}^{\ast})^{q-s-1}-d_{2,2}u_{1}^{s}(d_{2,2}^{\ast})^{q-s}}{u_{0}^{q}}\quad(\textrm{by }(\ref{eq:3.2}))\\
&=&\frac{(u_{0}^{q+1}+d_{2,2}d_{2,2}^{\ast})u_{1}^{s}(d_{2,2}^{\ast})^{q-s-1}-d_{2,2}u_{1}^{s}(d_{2,2}^{\ast})^{q-s}}{u_{0}^{q}}\\
&=&u_{0}u_{1}^{s}(d_{2,2}^{\ast})^{q-s-1}.
 \end{eqnarray*}
Similarly, one can prove the other identities.
\end{proof}

 We are now ready to construct a free $R_{2}^{SL_{2}^{2}}$-module basis for $R_{2}^{SL_{2}}$, i.e., to prove the following
theorem.

 \begin{theorem}\label{thm:3.3} The invariant ring
$R_{2}^{SL_{2}}$ is a free  $R_{2}^{SL_{2}^{2}}$-module on the following set
\begin{equation}
\mathds{S}=\left\{ \begin{aligned}
(u_{1}^{\ast})^{i}u_{1}^{j},&\quad 0\leq i,j\leq q-1;\\
(u_{1}^{\ast})^{i}u_{1}^{j}u_{0}^{k},&\quad 0\leq i,j\leq q-2;~~ 1\leq k\leq q;\\
h_{s}u_{0}^{k},&\quad 1\leq s\leq q-2;~~ 0\leq k\leq q-1
\end{aligned} \right\}.
\end{equation}
\end{theorem}

 \begin{proof}
Let $I_{P}$ and $I_{S}$ denote the homogeneous ideals generated by $\left\{d_{2,2},d_{2,2}^{\ast},c_{2,1},c_{2,1}^{\ast}\right\}$ in $R_{2}^{P_{2}}$ and $R_{2}^{SL_{2}}$ respectively. To prove that the $q(q^{2}-1)$ elements  in $\mathds{S}$ form a free basis, we have to verify two conditions:
\begin{enumerate}
  \item these elements do not belong to $I_{S}$; and
  \item they are linearly independent modulo $I_{S}$.
\end{enumerate}
Since $R_{2}^{SL_{2}}\subset R_{2}^{P_{2}}$, we have $I_{S}\subset I_{P}$. For the condition (1), it suffices to show that every element does not belong to $I_{P}$; namely, to show that $(u_{1}^{\ast})^{q-1}u_{1}^{q-1}$, $(u_{1}^{\ast})^{q-2}u_{1}^{q-2}u_{0}^{q}$ and $h_{s}u_{0}^{q-1}$ are not zero modulo $I_{P}$. We now are working on modulo  $I_{P}$, and we  check them one by one. First of all, it follows from (\ref{eq:3.5}) and (\ref{eq:3.6}) that
 \begin{eqnarray*}
(u_{1}u_{1}^{\ast})^{q-1}
&\equiv& (\phi_{1}\phi_{1}^{\ast})^{(q-1)^{2}}u_{0}^{2(q-1)}+u_{0}^{q-1}(\phi_{1}^{q(q-1)}(\phi_{2}^{\ast})^{q-1}+(\phi_{1}^{\ast})^{q(q-1)}\phi_{1}^{q-1})\\
&\equiv& (\phi_{1}\phi_{1}^{\ast})^{(q-1)^{2}}u_{0}^{2q-2}-2u_{0}^{q-1}(\phi_{1}\phi_{1}^{\ast})^{q(q-1)}.
 \end{eqnarray*}
By (\ref{eq:2.3}), we have
$(\phi_{1}\phi_{1}^{\ast})^{(q-1)^{2}}u_{0}^{2q-2}=(\phi_{1}\phi_{1}^{\ast})^{(q-1)^{2}}u_{0}^{q-2}((\phi_{1}\phi_{1}^{\ast})^{q-1}u_{0}+
\phi_{1}^{q}\phi_{2}^{\ast}+(\phi_{1}^{\ast})^{q}\phi_{2})
\equiv (\phi_{1}\phi_{1}^{\ast})^{q(q-1)}u_{0}^{q-1}.
$ Thus $(u_{1}u_{1}^{\ast})^{q-1}\equiv -(\phi_{1}\phi_{1}^{\ast})^{q(q-1)}u_{0}^{q-1}$ does not belong to $I_{P}$  because the latter is an element of the free basis $\mathds{P}$ by Proposition \ref{prop:2.4}. Secondly, from (\ref{eq:3.2}) we have $(u_{1}^{\ast})^{q-2}u_{1}^{q-2}u_{0}^{q+1}\equiv (u_{1}u_{1}^{\ast})^{q-1}$. Thus $(u_{1}^{\ast})^{q-2}u_{1}^{q-2}u_{0}^{q}$ is not an element in $I_{P}$. Finally, it follows from (\ref{eq:3.5}) and (\ref{eq:3.6})) that
$$h_{s}u_{0}^{q-1}\equiv \frac{\left(\phi_{1}^{\ast}\phi_{2}\right)^{s+1}(d_{2,2}^{\ast})^{q-s-1}+\left(\phi_{1}\phi_{2}^{\ast}\right)^{q-s}d_{2,2}^{s}}{u_{0}}.$$
By (\ref{eq:2.4}) and (\ref{eq:2.5}), we eventually deduce that $h_{s}u_{0}^{q-1}\equiv (-1)^{s}\phi_{2}^{s}(\phi_{2}^{\ast})^{q-s-1}u_{0}^{q-1}$, which is not in $I_{P}$. Thus the condition (1) is satisfied.

For the condition (2), we need only to show that $\left\{(u_{1}^{\ast})^{i}u_{1}^{j}\right\}$, $\left\{(u_{1}^{\ast})^{i}u_{1}^{j}u_{0}^{k}\right\}$, and $\left\{h_{s}u_{0}^{k}\right\}$ are all linearly independent modulo $I_{S}$. We are working on modulo  $I_{S}$.
First, we assume that all $(u_{1}^{\ast})^{i}u_{1}^{j}$  are linearly dependent.  Then there exists an equation $\sum a_{ij}(u_{1}^{\ast})^{i}u_{1}^{j}\equiv 0$ where $a_{ij}$ in $F_{q}$.
We  order the monomials $(u_{1}^{\ast})^{i}u_{1}^{j}$ lexicographically using the exponent sequences $(i,j)$.
Let $a_{i_{0}j_{0}}\neq 0$ be the coefficient of the smallest monomial. Multiplying the left hand side of the preceding equation by
 $(u_{1}^{\ast})^{q-1-i_{0}}u_{1}^{q-1-j_{0}}$, we obtain
\begin{eqnarray*}
0 \equiv\sum a_{ij}(u_{1}^{\ast})^{q-1+i-i_{0}}u_{1}^{q-1+j-j_{0}}=a_{i_{0}j_{0}}(u_{1}^{\ast})^{q-1}u_{1}^{q-1}+\sum_{i>i_{0}\textrm{ and }
 j>j_{0}} a_{ij}(u_{1}^{\ast})^{q-1+i-i_{0}}u_{1}^{q-1+j-j_{0}}.
 \end{eqnarray*}
Combining this equation with (\ref{eq:3.3}) and (\ref{eq:3.4}) we get $a_{i_{0}j_{0}}(u_{1}^{\ast})^{q-1}u_{1}^{q-1}\equiv 0$. Thus $a_{i_{0}j_{0}}$ equals zero, and this is a contradiction. Second, assume that there exists an equation $\sum a_{ijk}(u_{1}^{\ast})^{i}u_{1}^{j}u_{0}^{k}\equiv 0$ and let $a_{i_{0}j_{0}k}\neq 0$ be the coefficient of the smallest monomial in the lexicographical order. Multiplying the left hand side of this equation by $(u_{1}^{\ast})^{q-2-i_{0}}u_{1}^{q-2-j_{0}}$, we obtain the resulting sum:
 \begin{equation*}
(u_{1}^{\ast})^{q-2}u_{1}^{q-2}\left(\sum_{k} a_{i_{0}j_{0}k}u_{0}^{k}\right)+\textrm{ other higher terms with larger powers of } u_{1}^{\ast}\textrm{ and }u_{1}.
\end{equation*}
It follows from (\ref{eq:3.10}) and (\ref{eq:3.11}) that $u_{0}(u_{1}^{\ast})^{q-1}\equiv 0$ and $
u_{0}u_{1}^{q-1}\equiv 0.$ Thus
\begin{equation}\label{eq:3.13}
(u_{1}^{\ast})^{q-2}u_{1}^{q-2}\left(\sum_{k} a_{i_{0}j_{0}k}u_{0}^{k}\right)\equiv 0.
\end{equation}
Let $k_{0}$ is the least integer for which $a_{i_{0}j_{0}k_{0}}\neq 0$. We then multiply the left side of (\ref{eq:3.13}) by $u_{0}^{q+1-k_{0}}$.
Similarly, it follows from (\ref{eq:3.2}) that  $a_{i_{0}j_{0}k_{0}}(u_{1}^{\ast})^{q-2}u_{1}^{q-2}u_{0}^{q+1} \equiv 0$, which implies that $a_{i_{0}j_{0}k_{0}}=0$, a contraction. The same method can be applied to $\left\{h_{s}u_{0}^{k}\right\}$.
Assume that all $h_{s}u_{0}^{k}$  are linearly dependent  with a relation $\sum a_{sk}h_{s}u_{0}^{k}\equiv 0$.
This equation can be rewritten as
\begin{equation}\label{eq:3.14}
\left(\sum_{s} a_{sk_{1}}h_{s}\right)u_{0}^{k_{1}}+\cdots+\left(\sum_{s} a_{sk_{m}}h_{s}\right)u_{0}^{k_{m}}\equiv 0,
\end{equation}
 where $0\leq k_{1}<\cdots<k_{m}\leq q-1$.
Multiplying the left hand side of (\ref{eq:3.14}) by
 $u_{0}^{q-1-k_{1}}$, we obtain
  $\left(\sum_{s} a_{sk_{1}}h_{s}\right)u_{0}^{q-1}+\sum_{s}\sum_{t\geq q} a_{st}h_{s}u_{0}^{t}\equiv 0.$
By (\ref{eq:3.7}), we see that for $t\geq q$, every $h_{s}u_{0}^{t}$ belongs to $I_{S}$. Thus if we write $b_{s}$ for $a_{sk_{1}}$, then
$\sum_{s} b_{s}\left(h_{s}u_{0}^{q-1}\right)\equiv 0.$ It is easy to see that all $h_{s}u_{0}^{q-1}$ are linearly independent modulo $I_{S}$ since
all $\phi_{2}^{s}(\phi_{2}^{\ast})^{p-s-1}u_{0}^{q-1}$ are linearly independent modulo $I_{P}$ and $I_{S}\subset I_{P}$.
 \end{proof}

Now we can compute the Hilbert series of $R_{2}^{SL_{2}}$.
We recall that the
Hilbert series of a graded algebra $\mathfrak{A}=\bigoplus_{i=0}^{\infty} \mathfrak{A}_{i}$ over a field $k$ is defined to be   $H(\mathfrak{A},t)=\sum_{i=0}^{\infty} \textrm{dim}_{k} \mathfrak{A}_{i} t^{i}$. The Hilbert series of a  polynomial algebra
$N=F_{q}[f_{1},\cdots,f_{r}]$ with degrees $|f_{i}|=d_{i}$ is equal to
$H(N,t)=\prod_{i=1}^{r}\left(1-t^{d_{i}}\right)^{-1}.$
Moreover, if $M$ is a free $N$-module with $\{m_{1},\cdots,m_{s}\}$ as basis elements of degrees $|m_{j}|=e_{j}$, we have the Hilbert series
$$H(M,t)=\frac{\sum_{j=1}^{s}t^{e_{j}}}{\prod_{i=1}^{r}\left(1-t^{d_{i}}\right)}.$$
All elements and their degrees in the basis $\mathds{S}$ of $R_{2}^{SL_{2}}$ as a free $R_{2}^{SL_{2}^{2}}$-module are listed in the following Table 1.
\begin{center}\footnotesize{
\begin{longtable}{|l|l|l|l|l|l|l|l|}
\caption[{$\mathds{S}$-the second invariants of $SL_{2}$}]{{\rm $\mathds{S}$-the second invariants of $SL_{2}$}} \label{2.3} \\
 \hline
\endfirsthead

\multicolumn{8}{c}%
{\footnotesize  \tablename\ \thetable{}-- continued from previous page} \\
\hline
\endhead

\hline \multicolumn{8}{|r|}{{Continued on next page}} \\ \hline
\endfoot

\hline
\endlastfoot
 Basis & 1 & $u_{1}$ & $u_{1}^{2}$ & $u_{1}^{3}$& $\cdots$& $u_{1}^{q-1}$ \\ \hline
 Degree        & 0 & $q+1$   & $2(q+1)$    & $3(q+1)$   & $\cdots$& $(q-1)(q+1)$ \\
 \hline
 Basis &  & $u_{2}$ & $u_{2}^{2}$ & $u_{2}^{3}$& $\cdots$& $u_{2}^{q-1}$ \\ \hline
 Degree        &  & $q+1$   & $2(q+1)$    & $3(q+1)$   & $\cdots$& $(q-1)(q+1)$ \\
 \hline
 Basis &  & $u_{1}u_{2}$ & $u_{1}u_{2}^{2}$ & $u_{1}u_{2}^{3}$& $\cdots$& $u_{1}u_{2}^{q-1}$ \\ \hline
 Degree        &  & $2(q+1)$   & $3(q+1)$    & $4(q+1)$   & $\cdots$& $q(q+1)$ \\
 \hline
 Basis &  & $u_{1}^{2}u_{2}$ & $u_{1}^{2}u_{2}^{2}$ & $u_{1}^{2}u_{2}^{3}$& $\cdots$& $u_{1}^{2}u_{2}^{q-1}$ \\ \hline
 Degree        &  & $3(q+1)$   & $4(q+1)$    & $5(q+1)$   & $\cdots$& $(q+1)(q+1)$ \\
 \hline
\quad\quad\vdots   &  & \quad\vdots & \quad\vdots  & \quad\vdots &\vdots  &\quad\quad\vdots  \\ \hline
 Basis &  & $u_{1}^{q-2}u_{2}$ & $u_{1}^{q-2}u_{2}^{2}$ & $u_{1}^{q-2}u_{2}^{3}$& $\cdots$& $u_{1}^{q-2}u_{2}^{q-1}$ \\ \hline
 Degree        &  & $(q-1)(q+1)$   & $q(q+1)$    & $(q+1)(q+1)$   & $\cdots$& $(2q-3)(q+1)$ \\
 \hline
 Basis &  & $u_{1}^{q-1}u_{2}$ & $u_{1}^{q-1}u_{2}^{2}$ & $u_{1}^{q-1}u_{2}^{3}$& $\cdots$& $u_{1}^{q-1}u_{2}^{q-1}$ \\ \hline
 Degree        &  & $q(q+1)$   & $(q+1)(q+1)$    & $(q+2)(q+1)$   & $\cdots$& $(2q-2)(q+1)$ \\
 \hline \hline
 Basis & $u_{0}^{k}$ & $u_{1}u_{0}^{k}$ & $u_{1}^{2}u_{0}^{k}$ & $u_{1}^{3}u_{0}^{k}$& $\cdots$& $u_{1}^{q-2}u_{0}^{k}$ \\ \hline
 Degree        & $2k$ & $2k+q+1$   & $2k+2(q+1)$    & $2k+3(q+1)$   & $\cdots$& $2k+(q-2)(q+1)$ \\
 \hline
 Basis &  & $u_{2}u_{0}^{k}$ & $u_{2}^{2}u_{0}^{k}$ & $u_{2}^{3}u_{0}^{k}$& $\cdots$& $u_{2}^{q-2}u_{0}^{k}$ \\ \hline
 Degree&  & $2k+q+1$   & $2k+2(q+1)$    & $2k+3(q+1)$   & $\cdots$& $2k+(q-2)(q+1)$ \\
 \hline
 Basis &  & $u_{1}u_{2}u_{0}^{k}$ & $u_{1}u_{2}^{2}u_{0}^{k}$ & $u_{1}u_{2}^{3}u_{0}^{k}$& $\cdots$& $u_{1}u_{2}^{q-2}u_{0}^{k}$ \\ \hline
 Degree&  & $2k+2(q+1)$   & $2k+3(q+1)$    & $2k+4(q+1)$   & $\cdots$& $2k+(q-1)(q+1)$ \\
 \hline
 Basis &  & $u_{1}^{2}u_{2}u_{0}^{k}$ & $u_{1}^{2}u_{2}^{2}u_{0}^{k}$ & $u_{1}^{2}u_{2}^{3}u_{0}^{k}$& $\cdots$& $u_{1}^{2}u_{2}^{q-2}u_{0}^{k}$ \\ \hline
 Degree &  & $2k+3(q+1)$   & $2k+4(q+1)$    & $2k+5(q+1)$   & $\cdots$& $2k+q(q+1)$ \\
 \hline
\quad\quad\vdots   &  & \quad\vdots & \quad\vdots  & \quad\vdots &\vdots  &\quad\quad\vdots  \\ \hline
 Basis &  & $u_{1}^{q-2}u_{2}u_{0}^{k}$ & $u_{1}^{q-2}u_{2}^{2}u_{0}^{k}$ & $u_{1}^{q-2}u_{2}^{3}u_{0}^{k}$& $\cdots$& $u_{1}^{q-2}u_{2}^{q-2}u_{0}^{k}$ \\ \hline
 Degree&  & $2k+(q^{2}-1)$   & $2k+q(q+1)$    & $2k+(q+1)^{2}$   & $\cdots$& $2k+2q^{2}-2q-4$ \\
 \hline \hline
 Basis & $h_{1}$ & $h_{1}u_{0}$ & $h_{1}u_{0}^{2}$ & $h_{1}u_{0}^{3}$& $\cdots$& $h_{1}u_{0}^{q-1}$ \\ \hline
 Degree & $q^{2}-q$ & $q^{2}-q+2$   & $q^{2}-q+4$    & $q^{2}-q+6$   & $\cdots$& $q^{2}+q-2$  \\
 \hline
 Basis & $h_{2}$ & $h_{2}u_{0}$ & $h_{2}u_{0}^{2}$ & $h_{2}u_{0}^{3}$& $\cdots$& $h_{2}u_{0}^{q-1}$ \\ \hline
 Degree & $q^{2}-q$ & $q^{2}-q+2$   & $q^{2}-q+4$    & $q^{2}-q+6$   & $\cdots$& $q^{2}+q-2$  \\
 \hline
 \quad\quad\vdots   &  & \quad\vdots & \quad\vdots  & \quad\vdots &\vdots  &\quad\quad\vdots  \\ \hline
 Basis & $h_{q-2}$ & $h_{q-2}u_{0}$ & $h_{q-2}u_{0}^{2}$ & $h_{q-2}u_{0}^{3}$& $\cdots$& $h_{q-2}u_{0}^{q-1}$ \\ \hline
 Degree & $q^{2}-q$ & $q^{2}-q+2$   & $q^{2}-q+4$    & $q^{2}-q+6$   & $\cdots$& $q^{2}+q-2$ \\
 \hline
\end{longtable}
}\end{center}
Hence we get the Hilbert series of $R_{2}^{SL_{2}}$,
$$H\left(R_{2}^{SL_{2}},t\right)=
\frac{H_{1}(t)+\sum_{k=1}^{q}H_{2}(k,t)+H_{3}(t)}{(1-t^{q+1})^{2}(1-t^{q^{2}-q})^{2}},$$ where
\begin{eqnarray*}
H_{1}(t)&=&1+2t^{q+1}+3t^{2q+2}+\cdots+qt^{q^{2}-1}+(q-1)t^{q^{2}+q}+\cdots+t^{2q^{2}-2}\\
H_{2}(k,t)&=&t^{2k}+2t^{2k+q+1}+3t^{2k+2q+2}+\cdots+(q-1)t^{2k+q^{2}-q-2}\\
&&+(q-2)t^{2k+q^{2}-1}+\cdots+t^{2k+2q^{2}-2q-4}\\
H_{3}(t)&=&(q-2)t^{q^{2}-q}+(q-2)t^{q^{2}-q+2}+\cdots+(q-2)t^{q^{2}+q-2}.
\end{eqnarray*}

We conclude this section with an application of the Hilbert series. From a well-known theorem due to Stanley (\cite{Sta1978}, Theorem 4.4 or \cite{NS2002}, page 145), we deduce that the invariant ring $R_{2}^{SL_{2}}$ is a Gorenstein algebra if and only if
$H\left(R_{2}^{SL_{2}},\frac{1}{t}\right)=t^{i}H\left(R_{2}^{SL_{2}},t\right),$
for some integer $i$. Thus

\begin{theorem}\label{thm:3.4}
The invariant ring  $R_{2}^{SL_{2}}$ is a Gorenstein algebra.
\end{theorem}

\begin{proof}
Indeed,
\begin{eqnarray*}
H_{1}\left(\frac{1}{t}\right)&=&\frac{1}{t^{2q^{2}-2}}H_{1}(t)\\
\sum_{k=1}^{q}H_{2}\left(k,\frac{1}{t}\right)&=&\frac{1}{t^{2q^{2}-2}}\sum_{k=1}^{q}H_{2}(k,t)\\
H_{3}\left(\frac{1}{t}\right)&=&\frac{1}{t^{2q^{2}-2}}H_{3}(t).
\end{eqnarray*}
Then
\begin{eqnarray*}
H\left(R_{2}^{SL_{2}},\frac{1}{t}\right)&=&\frac{\frac{1}{t^{2q^{2}-2}}\left(H_{1}(t)+
\sum_{k=1}^{q}H_{2}(k,t)+H_{3}(t)\right)}{\frac{1}{t^{2q^{2}+2}}(1-t^{q+1})^{2}(1-t^{q^{2}-q})^{2}}\\
&=&\frac{t^{2q^{2}+2}}{t^{2q^{2}-2}}H\left(R_{2}^{SL_{2}},t\right)=t^{4}H\left(R_{2}^{SL_{2}},t\right),
\end{eqnarray*}as desired.
\end{proof}

\begin{exam}{\rm
We consider the invariant ring of $SL_{2}(F_{3})$. It follows from Theorem \ref{thm:3.3} that
$R_{2}^{SL_{2}(F_{3})}$ is a free  $R_{2}^{SL_{2}(F_{3})^{2}}$-module on the following set
\begin{equation}
\left\{ \begin{aligned}
u_{-1}^{i}u_{1}^{j},&\quad 0\leq i,j\leq 2;\\
u_{-1}^{i}u_{1}^{j}u_{0}^{k},&\quad 0\leq i,j\leq 1;~~ 1\leq k\leq 3;\\
h_{1}u_{0}^{k},&\quad 0\leq k\leq 2
\end{aligned} \right\}.
\end{equation}Thus
\begin{eqnarray*}
H_{1}(t)&=&1+2t^{4}+3t^{8}+2t^{12}+t^{16}\\
H_{2}(t)&=&\sum_{k=1}^{3}\left(t^{2k}+2t^{2k+4}+t^{2k+8}\right)\\
&=&t^{2}+t^{4}+3t^{6}+2t^{8}+3t^{10}+t^{12}+t^{14}\\
H_{3}(t)&=&t^{6}+t^{8}+t^{10}.
\end{eqnarray*}
The functions $H_{1}(t),H_{2}(t)$ and $H_{3}(t)$ are ``symmetric" in $t$
and $t^{-1}$. So
$H\left(R_{2}^{SL_{2}(F_{3})},t^{-1}\right)=t^{4}H\left(R_{2}^{SL_{2}(F_{3})},t\right)$ and $R_{2}^{SL_{2}(F_{3})}$ is a Gorenstein algebra by Stanley's theorem.\qed
}\end{exam}

\section{A Special Case of Bonnaf\'{e}-Kemper's Conjecture}
\setcounter{equation}{0}
\renewcommand{\theequation}
{4.\arabic{equation}}

\setcounter{theorem}{0}
\renewcommand{\thetheorem}
{4.\arabic{theorem}}

In this section, we shall apply the same techniques to the case of $GL_{2}$.
First we need the notion of  ``relative trace map" of a group relative  to its subgroup, which is  an important tool in calculating
the invariant ring. Let $H\leq G$ be two finite groups  acting linearly on the polynomial algebra $R$ and $\{g_{t}\}$ be a set of right representative for $H$ in $G$. If the index $[G:H]$ is invertible in the ground field, then
the \textit{relative trace map} $\textrm{Tr}_{H}^{G}: R^{H}\rightarrow R^{G}$ is defined by
\begin{equation}
f\mapsto \frac{1}{[G:H]}\sum g_{t}(f)
\end{equation}
for all $f\in R^{H}$. It is easy to verify that $\textrm{Tr}_{H}^{G}$ is an epimorphism of $R^{G}$-modules.

\begin{prop} The set $\mathds{G}=\left\{(d_{2,2}^{\ast})^{\alpha}d_{2,2}^{\beta}\cdot\mathds{S}, 0\leq \alpha,\beta\leq q-2\right\}$ is a free basis for
$R_{2}^{SL_{2}}$ over $R_{2}^{GL_{2}^{2}}$.
\end{prop}

\begin{proof}
It is easy to prove that $R_{2}^{GL_{2}^{2}}$ is also a polynomial algebra over $F_{q}$ and $\left\{(d_{2,2}^{\ast})^{\alpha}d_{2,2}^{\beta}, 0\leq \alpha,\beta\leq q-2\right\}$ is a free basis for $R_{2}^{SL_{2}^{2}}$ over
$R_{2}^{GL_{2}^{2}}$. By Theorem \ref{thm:3.3}, $R_{2}^{SL_{2}}$ is just the $R_{2}^{GL_{2}^{2}}$-module spanned by $\mathds{G}$.
By Lemmas \ref{lem:2.1} and  \ref{prop:2.2}, we know that
the rank of $R_{2}^{SL_{2}}$ as a free $R_{2}^{GL_{2}^{2}}$-module is $q(q^{2}-1)(q-1)^{2}$. This is just the size of $\mathds{G}$.
\end{proof}

\begin{theorem} \label{thm:4.2}
The invariant ring
$R_{2}^{GL_{2}}$ is a free  $R_{2}^{GL_{2}^{2}}$-module on the following set
\begin{equation}
\mathds{D}=\left\{ \begin{aligned}
(u_{1}^{\ast})^{i}u_{1}^{j}(d_{2,2}^{\ast}d_{2,2})^{\alpha},&\quad 0\leq i,j\leq q-1,~~0\leq \alpha\leq q-2,\\
(u_{1}^{\ast})^{i}u_{1}^{j}u_{0}^{k}(d_{2,2}^{\ast}d_{2,2})^{\alpha},&\quad 0\leq i,j\leq q-2,~~ 1\leq k\leq q,~~0\leq \alpha\leq q-2,\\
h_{s}u_{0}^{k}(d_{2,2}^{\ast})^{\beta+s}(d_{2,2})^{\beta},&\quad 1\leq s\leq q-2,~~ 0\leq k\leq q-1,~~0\leq \beta\leq q-2
\end{aligned} \right\}.
\end{equation}
\end{theorem}

\begin{proof}
We choose $\left\{g_{z}SL_{2}, z\in F_{q}^{\ast}\right\}$ as a set of right coset representatives for $SL_{2}$ in $GL_{2}$,
where $g_{z}=\left(\begin{smallmatrix}
z&0\\
0&1
\end{smallmatrix}\right)$. We write $\phi:=\textrm{Tr}_{SL_{2}}^{GL_{2}}$  for
the relative trace map  from $ R_{2}^{SL_{2}}$ to $R_{2}^{GL_{2}}$, then
\begin{equation}
\phi(f)=\frac{1}{[GL_{2}:SL_{2}]}\sum_{z\in F_{q}^{\ast}} g_{z}(f)=
-\sum_{z\in F_{q}^{\ast}} g_{z}(f)
\end{equation}
is an epimorphism of $R_{2}^{GL_{2}}$-modules. Since the set $\mathds{D}$ has exactly $(q^{2}-1)(q^{2}-q)$ elements, it suffices to prove that
the image of $\mathds{G}$ lies in the $R_{2}^{GL_{2}}$-module spanned by $\mathds{D}$.
We observe that $\phi$ fixes  $u_{1}^{\ast},u_{0},u_{1}$ and
$g_{z}(d_{2,2}^{\ast})=zd_{2,2}^{\ast},~g_{z}(d_{2,2})=\frac{1}{z}d_{2,2}, ~g_{z}(h_{s})=\frac{1}{z^{s}}h_{s}.$
Assume that $a$ is a positive integer and recall the following well-known identity (see (\cite{CHSW1996}, Lemma 9.4) for a proof)
$$\sum_{z\in F_{q}^{\ast}}z^{a}=\left\{\begin{array}{cc}
0,&a \not\equiv 0 ~~(\textrm{mod } q-1),\\
-1,&a \equiv 0 ~~(\textrm{mod } q-1).
\end{array}\right.$$
Hence
\begin{eqnarray*}
\phi\left((u_{1}^{\ast})^{i}u_{1}^{j}(d_{2,2}^{\ast})^{\alpha}d_{2,2}^{\beta}\right)&=&
-\sum_{z\in F_{q}^{\ast}}z^{\alpha-\beta}(u_{1}^{\ast})^{i}u_{1}^{j}(d_{2,2}^{\ast})^{\alpha}d_{2,2}^{\beta}\\
&=&-(u_{1}^{\ast})^{i}u_{1}^{j}(d_{2,2}^{\ast})^{\alpha}d_{2,2}^{\beta}\left(\sum_{z\in F_{q}^{\ast}}z^{\alpha-\beta}\right)\\
&=&\left\{\begin{array}{cc}
0,&\alpha-\beta \not\equiv 0 ~~(\textrm{mod } q-1),\\
(u_{1}^{\ast})^{i}u_{1}^{j}(d_{2,2}^{\ast})^{\alpha}d_{2,2}^{\beta},&\alpha-\beta \equiv 0 ~~(\textrm{mod } q-1).
\end{array}\right.\\
&=&\left\{\begin{array}{cc}
0,&\alpha\neq\beta,\\
(u_{1}^{\ast})^{i}u_{1}^{j}(d_{2,2}^{\ast})^{\alpha}d_{2,2}^{\alpha},&\alpha=\beta.
\end{array}\right.
\end{eqnarray*}
Similarly, we have
\begin{eqnarray*}
\phi\left((u_{1}^{\ast})^{i}u_{1}^{j}u_{0}^{k}(d_{2,2}^{\ast})^{\alpha}d_{2,2}^{\beta}\right)
&=&\left\{\begin{array}{cc}
0,&\alpha\neq\beta,\\
(u_{1}^{\ast})^{i}u_{1}^{j}u_{0}^{k}(d_{2,2}^{\ast})^{\alpha}d_{2,2}^{\alpha},&\alpha=\beta.
\end{array}\right.\\
\phi\left(h_{s}u_{0}^{k}(d_{2,2}^{\ast})^{\alpha}d_{2,2}^{\beta}\right)
&=&\left\{\begin{array}{cc}
0,&\alpha\neq\beta+s,\\
h_{s}u_{0}^{k}(d_{2,2}^{\ast})^{\beta+s}d_{2,2}^{\beta},&\alpha=\beta+s.
\end{array}\right.
\end{eqnarray*}
This  completes the proof.
\end{proof}

\begin{theorem}\label{thm:4.3}
The invariant ring $R_{2}^{GL_{2}}$ is a Gorenstein algebra generated by
$$\left\{d_{2,2}^{q-1},c_{2,1},(d_{2,2}^{\ast})^{q-1},c_{2,1}^{\ast},u_{1}^{\ast},u_{0},u_{1}\right\}.$$
\end{theorem}

\begin{proof}
The similar arguments as in Theorem \ref{thm:3.4} will imply that
$
H\left(R_{2}^{GL_{2}},\frac{1}{t}\right)
=t^{4}H\left(R_{2}^{GL_{2}},t\right).
$
Stanley's theorem yields that $R_{2}^{GL_{2}}$ is a Gorenstein algebra.
Now we discuss the generators of $R_{2}^{GL_{2}}$.
We employ $R_{2}^{\ast}$ to denote the $F_{q}$-algebra generated by $\left\{d_{2,2}^{q-1},c_{2,1},(d_{2,2}^{\ast})^{q-1},c_{2,1}^{\ast},u_{1}^{\ast},u_{0},u_{1}\right\}$, and we claim that $R_{2}^{GL_{2}}=R_{2}^{\ast}$.
It follows from (\ref{eq:3.2}) that $d_{2,2}^{\ast}d_{2,2}=u_{1}^{\ast}u_{1}-u_{0}^{q+1}\in R_{2}^{\ast}$.
Since $R_{2}^{GL_{2}^{2}}=F_{q}\left[d_{2,2}^{q-1},c_{2,1},(d_{2,2}^{\ast})^{q-1},c_{2,1}^{\ast}\right]$, and by Theorem \ref{thm:4.2},  it suffices to prove that
each $h_{s}u_{0}^{k}(d_{2,2}^{\ast})^{\alpha}d_{2,2}^{\beta}$ in $\mathds{D}$ is contained in $R_{2}^{\ast}$.
By (\ref{eq:3.8}) we get
$$
h_{s}(d_{2,2}^{\ast})^{s}=\frac{(c_{2,1}^{\ast})u_{0}^{q}u_{1}^{s}+(u_{1}^{\ast})^{q-s}\left(\sum_{k=0}^{s-1}(-1)^{s-k}{s\choose k}(u_{1}^{\ast}u_{1})^{k}\left(u_{0}^{q+1}\right)^{s-k}\right)}{u_{0}^{q}}\in R_{2}^{\ast}.
$$
Thus
$
h_{s}u_{0}^{k}(d_{2,2}^{\ast})^{\beta+s}d_{2,2}^{\beta}=(h_{s}(d_{2,2}^{\ast})^{s})u_{0}^{k}(d_{2,2}^{\ast}d_{2,2})^{\beta}\in R_{2}^{\ast},
$ as desired.
\end{proof}

We close this paper with two remarks.

\begin{rem}{\rm
When the manuscript was finished, we sent a copy to David Wehlau. He tell us that there maybe exist another way to  compute $R_{2}^{GL_{2}}$. Here we record his idea:
First, one should compute the generators and their relations for $R_{2}^{SL_{2}}$
(see (\ref{eq:3.2})-(\ref{eq:3.4}) and (\ref{eq:3.9})-(\ref{eq:3.11}) for the required relations),  and notice that the factor group $GL_{2}/SL_{2}$ is isomorphic to $F_{q}^{\ast}$. Thus one can
compute $R_{2}^{GL_{2}}$ via
$R_{2}^{GL_{2}}=\left(R_{2}^{SL_{2}}\right)^{F_{q}^{\ast}}.$ However we do not know which method is faster to compute this invariant ring.\qed
}
\end{rem}

\begin{rem}{\rm
In \cite{BK2011}, Bonnaf\'{e} and  Kemper showed that for $G\in \left\{P_{2}(F_{q}),U_{2}(F_{q})\right\}$,
the invariant ring  $F_{q}[V\oplus V^{\ast}]^{G}$ is generated by generators of $F_{q}[V]^{G}$,
invariants of the form $u_{i}$, and  their $\ast$-images (actually, they proved that this result holds for all $n$ = dim$V\geq2$).
Theorem \ref{thm:4.3} means that a similar result is also valid for $GL_{2}(F_{q})$.
Furthermore, they found that this statement is not valid for $SL_{2}(F_{3})$, i.e.,
$F_{3}[V\oplus V^{\ast}]^{SL_{2}(F_{3})}$ can not be generated by generators of $F_{3}[V]^{SL_{2}(F_{3})}$,
invariants of the form $u_{i}$, and their $\ast$-images. Here we can explain the reason.  Indeed, by Theorem \ref{thm:3.3} (or see Theorem 8.1 in \cite{CSW2010}), we see that
$$F_{3}[V\oplus V^{\ast}]^{SL_{2}(F_{3})}=F_{q}\left[d_{2,2},c_{2,1},d_{2,2}^{\ast},c_{2,1}^{\ast},u_{1}^{\ast},u_{0},u_{1},h_{1}\right],$$ and
 $h_{1}$ is not contained in
the subalgebra $F_{q}\left[d_{2,2},c_{2,1},d_{2,2}^{\ast},c_{2,1}^{\ast},u_{1}^{\ast},u_{0},u_{1}\right]$.\qed
}
\end{rem}

\section*{\textit{Acknowledgments}}
The author would like to thank Professors Hang Gao, Manfred Hartl,  and  David Wehlau for their help.

\end{document}